\documentclass[a4paper, 12pt]{amsart}

\usepackage{amsmath,amsthm,amssymb}
\usepackage{enumerate}
\usepackage{lipsum}
\usepackage{microtype} 
\usepackage[british]{babel}
\newtheorem{thm}{Theorem}[section]

\newtheorem{lem}[thm]{Lemma}

\theoremstyle{definition}

\theoremstyle{remark}

\numberwithin{equation}{section}

\begin{document}
\title[Solitary Waves for CNLS System with power nonlinearities]{
Stability of solitary-wave solutions of
                              coupled NLS equations with power-type nonlinearities}

\author[Santosh Bhattarai]{Santosh Bhattarai}
\address{Trocaire College,
360 Choate Ave, Buffalo, NY 14220 USA}
\email{sntbhattarai@gmail.com, bhattarais@trocaire.edu}

\begin{abstract}
This paper proves existence and stability results of solitary-wave solutions of a
system of 2-coupled nonlinear Schr\"{o}dinger equations
with power-type nonlinearities arising in several models of modern physics.
The existence of vector solitary-wave solutions (i.e, both components are nonzero)
is established via variational methods.
The set of minimizers is shown to be stable and further
information about the structures of this set are given.
The results extend stability results previously obtained by Cipolatti and Zumpichiatti \cite{[Cip]}, Nguyen and
Wang \cite{[Ngu],[Ngu4]}, and Ohta \cite{[Ohta]}.

\end{abstract}

\maketitle

\section{Introduction}
The nonlinear Schr\"{o}dinger (NLS) equation
\begin{equation}\label{nls}
iu_{t}+u_{xx}+|u|^{p-1}u=0,
\end{equation}
where $u$ is a complex-valued function of $(x,t)\in \mathbb{R}^{2},$
has been widely recognized as a universal mathematical model for
describing the evolution of a slowly varying wave packet in a general nonlinear wave system.
It plays an important role in a wide range of physical subjects such as
plasma physics \cite{[Dod]}
, nonlinear optics \cite{[Agr]}, hydrodynamics \cite{[Zak]}, magnetic systems \cite{[dem]}, to name a few.
The NLS equation has been also derived as the modulation equation for
wave packets in spatially periodic media such as photonic band gap materials and
Bose-Einstein condensates \cite{[Dal], [deS]}.

\smallskip

In certain physical situations, when there are two wavetrains moving with nearly the same group velocities, their interactions are then governed by the coupled NLS equations \cite{[Ros], [Yan]}.
For example, the coupled NLS systems appear in the study of
interactions of waves with different polarizations \cite{[Ber]}, the description of
nonlinear modulations of two monochromatic waves \cite{[New]}, the interaction of Bloch-wave packets in a periodic system \cite{[Shi]},
the evolution of two orthogonal pulse envelopes in birefringent optical fiber \cite{[Men]},
the evolution of two surface wave packets in deep water \cite{[Ros]}, to name a few.
The motivation for studying the coupled NLS systems also come from their applications in the Hartree-Fock theory for a double condensate, i.e., a binary mixture of Bose-Einstein condensates in two different hyperfine states \cite{[Esr]}.

\smallskip

 In this paper we consider
the following system of coupled 1-dimensional time-dependent nonlinear Schr\"{o}dinger equations:
\begin{equation}\label{CNLS}
\left\{
\begin{aligned}
iu_{t}+u_{xx}+(\alpha|u|^{p-2}+\tau|v|^{q}|u|^{q-2})u&=0 \\
iv_{t}+v_{xx}+(\beta |v|^{r-2}+\tau|u|^{q}|v|^{q-2})v&=0,
\end{aligned}
\right.
\end{equation}
where $u, v$ are complex-valued functions
of the real variables $x$ and $t,$ and the constants $\alpha, \beta, \tau$ are real.

\bigskip

The energy $H$ and the component mass $Q$ for the system \eqref{CNLS} are defined, respectively, as
\begin{equation}\label{energydef}
H(u,v)= |u_{x}|_{2}^{2}+|v_{x}|_{2}^{2}-\left(a|u|_{p}^{p}+b|v|_{r}^{r}+c|uv|_{q}^{q}\right),
\end{equation}
\begin{equation}
Q(u)=|u|_{2}^{2},
\end{equation}
and
\begin{equation}\label{Qdef}
Q(v)=|v|_{2}^{2},
\end{equation}
where $a=2\alpha/p,\ b=2\beta/r,$ and $c=2\tau/q.$ The conservation of these functionals is an important ingredient in our stability analysis.
(Here $|\cdot |_{p}$ denote the $L^{p}$ norm of complex-valued measurable functions on the line. For more details on our notation, see below.)

\smallskip

Solitary-wave solutions of
\eqref{CNLS} are, by definition,
solutions of the form
\begin{equation} \label{SO}
\begin{aligned}
&u(x,t)=e^{i(\omega _{1}-\sigma ^{2})t+i\sigma x +i\lambda_1}\Phi (x-2\sigma t),\\
&v(x,t)=e^{i(\omega _{2}-\sigma ^{2})t+i\sigma x + i\lambda_2}\Psi (x-2\sigma t),
\end{aligned}
\end{equation}
where $\omega_{1}, \omega_{2}, \sigma \in \mathbb{R},$ and $\Phi, \Psi:\mathbb{R}\to \mathbb{C}$ are functions of one variable whose values are small when $|\xi|=|x-2\sigma t|$ is large.
Notice that if we insert \eqref{SO} into \eqref{CNLS}, we see that $(\Phi, \Psi)$ solves the following system of ordinary differential equations
\begin{equation}\label{ODE}
\left\{
\begin{aligned}
-\Phi ^{\prime \prime }+\omega_{1}\Phi &=\alpha|\Phi|^{p-2}\Phi+\tau|\Psi|^{q}|\Phi|^{q-2}\Phi, \\
-\Psi ^{\prime \prime }+\omega_{2}\Psi &=\beta|\Psi|^{r-2}\Psi+\tau|\Phi|^{q}|\Psi|^{q-2}\Psi.
\end{aligned}
\right.
\end{equation}
The special case of \eqref{SO} when $\sigma=\lambda_1=\lambda_2=0,$ solutions of the form
\begin{equation}\label{SWave}
(u(x,t),v(x,t))=(e^{i\omega_{1}t}\Phi_{\omega_1}(x),e^{i\omega_{2}t}\Psi_{\omega_2}(x)),
\end{equation}
are usually referred
as standing-wave solutions. It is easy to see that $(u,v)$ as defined in \eqref{SWave} is a solution of \eqref{CNLS}
if and only if $(\Phi_{\omega_1},\Psi_{\omega_2})$ is a critical point for the functional $H(u,v),$
when $u$ and $v$ are varied subject to the constraints that $Q(u)$ and $Q(v)$ be held constant.
If $(\Phi_{\omega_1},\Psi_{\omega_2})$ is not only a
critical point, but in fact a global minimizer of the constrained variational problem for $H(u,v),$ then \eqref{SWave} is
called a ground-state solution of \eqref{CNLS}. In some cases, namely when $p=r=2q=4$ and
under certain conditions on $\alpha, \beta,$ and $\tau,$
it is
possible to show further that the ground-state solutions are solitary waves with the usual sech-profile (see, for example, \cite{[Ohta], [Ngu]}).

\smallskip

Over the past ten years, the existence of nontrivial solutions of the elliptic system \eqref{ODE} has been investigated by
many authors using different methods. In the case of a positive coupling
parameter $\tau,$ Maia et al. \cite{[AM4]} studied the existence result for positive solutions of \eqref{ODE} using
constrained minimization methods. They proved the existence of vector ground states of \eqref{ODE} i.e.,
minimal action solutions $(\Phi,\Psi)$ with both $\Phi, \Psi$ nontrivial.
Moreover, the authors gave sufficient
conditions
for ground states to be positive in both components which basically require the coupling
parameter $\tau$ to be positive and sufficiently large.
Also, Ambrosetti and Colorado \cite{[AM1]} and de Figueiredo and Lopes \cite{[deF]} have proved
the additional sufficient
conditions for the existence of positive ground-state solutions in the special case $p=r=2q=4.$
Furthermore, for $p=r=2q=4$ and small positive values of $\tau,$ Lin and Wei \cite{[AM3]} and Sirakov \cite{[AM5]}
proved the existence
of positive solutions which have minimal energy among all fully nontrivial solutions.
In the repulsive case $\tau<0,$
Mandel \cite{[AM6]} recently
established existence and nonexistence results concerning fully nontrivial minimal energy solutions.
In all these papers, the analysis of their constrained minimization
problems does not establish the stability property of solutions.
In order to
study the stability questions, one has to tackle a different variational
formulation.


\smallskip

Our aim here is to prove the stability of vector solitary-wave solutions of the coupled
nonlinear Schr\"{o}dinger system \eqref{CNLS}. The extensive mathematical literature on the
subject of stability of solitary waves began with
the work Benjamin \cite{[Benj]} (see also Bona \cite{[Bona]}) for the KdV equation.
In subsequent works, many techniques have been developed to refine and extend
Benjamin's original conception in many ways to include numerous equations and systems
such as Benjamin-Ono equation, intermediate long wave equation, nonlinear
Schr\"{o}odinger equation, Boussinesq systems, etc.
For instance, Cazenave and Lions \cite{[CL]} developed a method to prove
existence and stability of solitary waves when they are minimizers of the energy
functional and when a compactness condition on minimizing sequences
holds. Using the concentration compactness principle of Lions \cite{[L]}, they proved that the solution of \eqref{nls} of the form
$e^{i\omega t}\Phi(x),\ \omega>0,$ and $\Phi(x)$ real and positive,
is stable if $p<5$ in the following sense,
for every $\epsilon >0$ there
exists $\delta>0$ such that if $u_0\in H^{1}(\mathbb{R})$ satisfies
$\|u_0-\Phi\|_{H^{1}(\mathbb{R})}<\delta,$ then
the solution $u(t)$ of \eqref{nls} with $u(0)=u_0$ exists for all $t$ and
\begin{equation*}
\sup_{t\in \mathbb{R}} \inf_{\theta \in \mathbb{R}}\inf_{y \in \mathbb{R}} \|u(t)-e^{i\theta}\Phi(\cdot - y)\|_{H^{1}(\mathbb{R})}<\epsilon.
\end{equation*}
On the other hand, it was shown that solution of the form $e^{i\omega t}\Phi(x)$ for the equation \eqref{nls}
is unstable for any $\omega>0$ if $p\geq 5$ (see Berestycki and Cazenave \cite{[Bere]} for $p>5,$ and Weinstein \cite{[Wei]} for $p=5).$
The Cazenave and Lions
method has since been adapted by different authors to prove existence and
stability results of a variety of nonlinear
dispersive equations (see, for example, \cite{[AAu], [AB11], [SB1], [Cip], [Ngu], [Ohta]}).

\smallskip

We now present a brief discussion of what is currently known about the stability of solitary-wave solutions for \eqref{CNLS}.
In the special case $p=r=2q=4, \alpha=\beta>-1,$ and $\tau=1,$ (also known as the symmetric case),
the coupled nonlinear Schr\"{o}odinger system \eqref{CNLS} is known to have explicit solitary-wave solutions of the form (see, for example, \cite{[New], [Wad]})
\begin{equation}\label{OhtaSO}
(u_\Omega,v_\Omega)=(e^{i(\Omega-\sigma ^{2})t+i\sigma x + i\lambda_1}\phi_\Omega (x-2\sigma t),
e^{i(\Omega-\sigma ^{2})t+i\sigma x +i\lambda_2}\phi_\Omega (x-2\sigma t)),
\end{equation}
where $\Omega>0, \sigma, \lambda_1,$ and $\lambda_2$ are real constants, and
\begin{equation}\label{SOform}
\phi_\Omega(x)=\sqrt{\frac{2\Omega}{\alpha +1}}\text{sech}(\sqrt{\Omega} x).
\end{equation}
This solution describes a 2-component solitary-wave solutions with
the components of equal amplitude. It corresponds to a straight line $\omega_1=\omega_2$
in the parameter plane $(\omega_1, \omega_2)$ of a general two-parameter family of solitary waves of \eqref{CNLS}.
For this particular form of solitary wave, stability was proved by
Ohta \cite{[Ohta]}.
In \cite{[Ngu]}, the stability result in \cite{[Ohta]} was extended to include a more general setting. Namely, when $p=r=2q=4,$ and $0< \tau < \min\{\alpha,\beta\};$ or $\tau > \max\{\alpha,\beta\}$ and $\tau^{2}> \alpha \beta,$ they proved the stability of solitary-wave solutions of the form
\begin{equation*}
\begin{aligned}
&u_\Omega(x,t)=e^{i(\Omega-\sigma ^{2})t+i\sigma x+i\lambda_1 }\sqrt{\frac{\tau-\beta}{\tau^2-\alpha\beta}}\phi_{\Omega} (x-2\sigma t),\\
&v_\Omega(x,t)=e^{i(\Omega -\sigma ^{2})t+i\sigma x+i\lambda_2}\sqrt{\frac{\tau-\alpha}{\tau^2-\alpha\beta}}\phi_{\Omega} (x-2\sigma t),
\end{aligned}
\end{equation*}
where $\Omega>0, \sigma, \lambda_1,$ and $\lambda_2$ are real constants, and $\phi_{\Omega}$ as defined in \eqref{SOform}.
In \cite{[Cip]} and \cite{[deF]}, the stability results were proved by considering different variational settings than the one used in \cite{[Ngu]}.
For example, in \cite{[Cip]}, the authors considered the variational problem of finding minimizers of $H$ subject to one constraint being
the sum of $L^{2}-$ norms of the two components. This variational problem can have different solitary-wave
solutions. In fact, the last two pages of \cite{[Cip]} show that in the case when
\begin{equation*}
\alpha=\beta=\sqrt{\frac{\tau-\beta}{\tau^2-\alpha\beta}} \ \text{and} \ \tau < \sqrt{\frac{\tau-\beta}{\tau^2-\alpha\beta}},
\end{equation*}
the solitary-waves which solve the variational problem in \cite{[Ngu]} are not the
same as the solitary waves which solve the variational problem in \cite{[Cip]}.

\smallskip

In all papers mentioned in the preceding paragraph, the stability results were proved by using variational methods
in which
constraint functionals
were not independently chosen.
It is not clear whether the sets of solitary waves obtained from these papers
constitute a true two-parameter family of disjoint sets.
To obtain a true two-parameter family of solitary waves, one has to
characterize solitary waves as minimizers of the energy
functional subject to two independent constraints.
In \cite{[AB11]}, the authors proved existence of a true two-parameter
family of solitary waves in the context of NLS-KdV system, improving the existence result obtained previously in \cite{[Dias]}.
Their method also lead to the stability property of solitary waves.
Recently, following the same arguments used in \cite{[AB11]}, Nguyen and Wang \cite{[Ngu4]}
proved the stability of a two-parameter family of solitary waves
for the NLS system \eqref{CNLS} in the special case $p=r=2q=4.$
Here we are able to
prove existence and stability of solitary-wave
solutions of \eqref{CNLS} for all $\alpha, \beta, \tau >0,$ and for the range $2 < p,r, 2q < 6.$
We will follow the arguments used in \cite{[AB11]} to solve
a constrained minimization problem.
This approach allows us to obtain existence and stability results concerning a true two-parameter family of
solitary waves with both component positive, i.e., each component is of the form
$e^{i\theta t} p(x)$ with $\theta \in \mathbb{R}$ and $p(x)$ a real-valued positive function in $H^1(\mathbb{R}).$

\smallskip

Logically, prior to a discussion of stability in terms of perturbations of the initial data
should be a theory for the initial-valued problem itself. This issue has been
studied in \cite{[Ngu3]} (see also \cite{[Ca]}). It is proved in \cite{[Ngu3]} that for the range
$2<p,r,2q<6,$ for any $(u(x,0),v(x,0))\in Y,$ there exists a unique solution $(u(x,t),v(x,t))$
of one-dimensional coupled NLS system \eqref{CNLS} in $C(\mathbb{R},Y)$ emanating
from $(u(x,0),v(x,0)),$ and $(u(x,t),v(x,t))$ satisfies
\begin{equation*}
Q(u(x,t))=Q(u(x,0)), \ \ Q(v(x,t))=Q(v(x,0)),
\end{equation*}
and
\begin{equation*}
H(u(x,t),v(x,t)) =H(u(x,0),v(x,0)).
\end{equation*}
However there are some restriction on the applicable range
of $p,r,q$ in higher dimension (See \cite{[Ngu3],[Ca]} for more details).

\smallskip

We now describe briefly our results. The existence of solitary waves is obtained by
studying constrained minimization problem and applying the concentration-compactness lemma of P.~L.~Lions \cite{[L]}.
More precisely, for $s>0$ and $t>0,$ we define
\begin{equation}
\Sigma_{s,t}=\left\{(f,g)\in Y: Q(f)=s, Q(g)=t \right\}.
\end{equation}
and consider the problem of finding minimizers of the functional $H(f,g)$ subject to $(f,g) \in \Sigma_{s,t}.$
To prevent dichotomy of minimizing sequences while applying concentration-compactness method, one require to prove
the strict subadditivity of the variational problem
with respect to the constraint parameters. More precisely,
we require to prove strict subadditivity of the function
\begin{equation}
\label{Ist}
\Theta(s,t)=\inf \left\{ H(f,g): (f,g) \in \Sigma_{s,t} \right\}.
\end{equation}
We establish the
strict subadditivity of $\Theta(s,t)$ following the ideas and
results contained in \cite{[AB11]}, which utilize the fact that the $H^{1}-$norms
of some functions are
strictly decreasing when the mass of the functions are symmetrically rearranged. The set of minimizers, namely
\begin{equation}
\mathcal{F}_{s,t}=\left\{(\Phi,\Psi) \in Y: H(\Phi,\Psi)=\Theta(s,t), (\Phi,\Psi)\in \Sigma_{s,t} \right\}.
\end{equation}
is shown to be stable in the sense that a solution which starts near the set will remain near it for all time. We also
consider the question about the characterization of the set $\mathcal{F}_{s,t}.$

\smallskip

The following are our existence and stability results.
\begin{thm}
\label{existence}
Suppose $\alpha, \beta, \tau >0$ and $2<p, r, 2q<6.$

\begin{enumerate}

\item[(a)]
The function $\Theta(s,t)$ defined in \eqref{Ist} is finite, and if $\{(f_{n},g_{n})\}$ is any sequence in $Y$ such that
\begin{equation}\label{mindef}
\lim_{n\to \infty }Q(f_{n})=s,\ \lim_{n\to
\infty }Q(g_{n})=t,\ \text{ and } \ \lim_{n\to
\infty }H(f_n,g_n)=\Theta(s,t),
\end{equation}
then there exists a subsequence $\{(f_{n_{k}},g_{n_{k}})\}$ and a family
$\{y_{k}\}\subset \mathbb{R}$ such that $\{(f_{n_{k}}(\cdot
+y_{k}),g_{n_{k}}(\cdot +y_{k})\}$ converges strongly in $Y$ to
some $(\Phi,\Psi)$ in $\mathcal{F}_{s,t}$.
 In particular, the set
$\mathcal{F}_{s,t}$ is non-empty.

\smallskip

\item[(b)] Each pair
$(\Phi,\Psi) \in \mathcal{F}_{s,t}$ is a solution of \eqref{ODE}
for some $\omega_{1}>0$ and $\omega_{2}>0$, and thus when inserted into \eqref{SO} yields a two-parameter solitary-wave
solution to the system \eqref{CNLS}.

\smallskip

\item[(c)]
For each pair $(\Phi,\Psi)$ in $\mathcal{F}_{s,t}$, there exist numbers
$\theta_{1}, \theta_{2} \in \mathbb{R}$ and functions $\tilde \phi$ and $\tilde \psi$ such that $\tilde \phi(x),
\tilde \psi(x)>0$ for all $x \in \mathbb R$, and
\begin{equation*}
\Phi(x)=e^{i\theta_{1}}\tilde\phi(x)\ \text{and}\ \Psi(x)=e^{i\theta_{2}}\tilde\psi(x).
\end{equation*}
Moreover, the functions
$\Phi$ and $\Psi$ are infinitely differentiable on $\mathbb R$.

\smallskip

\item[(d)] For any $s,t>0,$ the set $\mathcal{F}_{s,t}$ is stable in the following sense: for every $\epsilon >0$, there
exists $\delta>0$ such that if $(u_0 ,v_0 )\in Y$ satisfies
\begin{equation*}
\inf\{\|(u_0,v_0)-(\Phi ,\Psi)\|_{Y}:(\Phi,\Psi)\in \mathcal{F}_{s,t}\}<\delta ,
\end{equation*}
and $(u(x,t),v(x,t))$ is the solution
of \eqref{CNLS}
with $(u(x,0),v(x,0)) =(u_0,v_0),$
 then for all $t \ge 0$,
\begin{equation*}
\inf\{\|(u(\cdot,t),v(\cdot,t))-(\Phi,\Psi)\|_{Y}:(\Phi,\Psi)\in \mathcal{F}_{s,t}\}<\epsilon.
\end{equation*}
\end{enumerate}
 \end{thm}

The method presented in this paper should be easily extendable to versions of \eqref{CNLS}
with combined power-type nonlinearities, such as the following system of coupled
nonlinear Schr\"{o}dinger equations
\begin{equation}\label{SongCNLS}
\left\{
\begin{aligned}
iu_{t}+u_{xx}+\alpha|u|^{p-2}u+\sum_{k=1}^{m} \tau |v|^{q_k}|u|^{q_k-2}u&=0 \\
iv_{t}+v_{xx}+\beta |v|^{r-2}v+ \sum_{k=1}^{m} \tau |u|^{q_k}|v|^{q_k-2}v&=0.
\end{aligned}
\right.
\end{equation}
The global existence of the solutions of this system is studied in \cite{[Song1]}. The energy
functional $K$ defined by
\begin{equation*}
K(u,v)=\frac{1}{2}\left(|u_x|_{2}^2+|v_x|_{2}^2\right) - \frac{1}{p}\left(\alpha|u|_{p}^p+\beta|v|_{r}^r+\tau \sum_{k=1}^{m}|uv|_{q_k}^{q_k}\right)
\end{equation*}
is conserved for the flow defined by \eqref{SongCNLS}. The functionals $Q(u)$ and $Q(v)$ defined above are conserved functionals for
\eqref{SongCNLS} as well. Our method can be applied to prove
an analogue of Theorem~\ref{existence} concerning existence and stability results of vector solitary-wave solutions to \eqref{SongCNLS}
for all $\alpha, \beta, \tau >0,$ and all $2<p, r, 2q_k<6 \ (k=1,2, . . .,m).$

\smallskip

\textit{Notation}.
For $1\leq s\leq \infty ,$ the space of complex measurable functions whose $s-$th
power is integrable will be denoted by
$L^{s}=L^{s}(\mathbb{R})$ and its standard norm by $\left\vert f\right\vert _{s},$
\begin{equation*}
\left\vert f\right\vert _{s}=\left( \int_{-\infty }^{\infty }\left\vert
f\right\vert ^{s}dx\right) ^{1/s}\textrm{ \ for }1\leq s<\infty
\end{equation*}
and $\left\vert f\right\vert _{\infty }$ is the essential supremum of $%
\left\vert f\right\vert $ on $\mathbb{R}.$ We denote
by $H^{1}(\mathbb{R})$ the Sobolev space of all complex-valued, measurable functions
defined on $\mathbb{R}$ such that both $f$ and
$f^{\prime}$ are in $L^{2}.$ The norm $\|.\|_{1}$ on $H^{1}$ is defined by
 \begin{equation*}
 \|f\|_{1}=\left(\int_{-\infty }^{\infty }(|f|^{2}+|f^{\prime}|^{2})\right)^{1/2}.
 \end{equation*}
In particular, we use $\left\Vert
f\right\Vert $ to denote the $L^{2}$ norm of a
function $f.$ We
define the space $Y$ to be the Cartesian product $H^{1}(\mathbb{R})\times H^{1}(%
\mathbb{R}),$ furnished with the norm
\begin{equation*}
\|(f,g)\|_{Y}^{2}=\|f \|_{1}^{2}+\|g \|_{1}^{2}.
\end{equation*}
 The letter $C$ will denote various positive constants whose exact values may change
from line to line but are not essential in the course of the analysis.

\section{Existence and Stability Results}

We assume throughout this paper, unless otherwise stated, that the assumptions $\alpha,\beta,\tau>0,$ and $2 < p, r, 2q <6$ hold.

\smallskip

To each minimizing sequence $\{(f_{n},g_{n})\}$ of $\Theta(s,t),$ we associate a sequence of nondecreasing functions
$M_{n}:[0,\infty)\to [0,s+t]$ defined by
\begin{equation*}
M_n(\zeta)=\sup_{y\in \mathbb{R}}\int_{y-\zeta}^{y+\zeta}\rho _{n}(x)\ dx.
\end{equation*}
where $\rho _{n}(x):=|f_{n}(x)|^2 + |g_{n}(x)|^2.$
An elementary argument (by Helly's selection theorem, for example) shows that any uniformly bounded sequence of nondecreasing functions on
$[0,\infty)$ must have a subsequence which converges pointwise (in fact, uniformly on compact sets)
to a nondecreasing limit function on $[0,\infty).$
Thus, $M_{n}(\zeta)$
has such a subsequence (see Lemma~\ref{Ibounded} below), which we again denote by $M_n.$
Let $M(\zeta):[0,\infty)\to [0,s+t]$ be the nondecreasing function to which $M_n$ converges, and define
\begin{equation}
\label{defgamma} \gamma =\lim_{\zeta\to \infty }M(\zeta).
\end{equation}
Then $\gamma$ satisfies $0\leq \gamma \leq s+t$. From Lions' Concentration Compactness Lemma (see \cite{[L]}), there
are three possibilities for the value of $\gamma$ that correspond to three
distinct types of limiting behavior of the sequence $\rho _{n}(x)$ as $%
n\to \infty ,$ which are suggestively labeled by Lions as
`vanishing', `dichotomy' and `compactness', respectively:
\begin{itemize}
\item[$(a)$] Case $1:($\textit{Vanishing}) $\gamma=0.$ Since $M(\zeta)$ is non-negative
and nondecreasing, this is equivalent to saying%
\begin{equation*}
M(\zeta)=\lim_{n\to \infty }M_{n}(\zeta)=\lim_{n\to \infty
}\sup_{y\in \mathbb{R}}\int_{y-\zeta}^{y+\zeta}\rho _{n}(x)\ dx=0,
\end{equation*}
for all $\zeta<\infty ,\ $or
\item[$(b)$] Case $2:($\textit{Dichotomy}) $\gamma \in (0,s+t),\ $or
\item[$(c)$] Case $3:($\textit{Compactness}) $\gamma=s+t,$ that is, there exists $%
\{y_{n}\}\subset \mathbb{R}$ such that $\rho _{n}(.+y_{n})$ is tight,
namely, for all $\varepsilon >0,$ there exists $\zeta<\infty $ such that for all $n \in \mathbb{N},$%
\begin{equation*}
\int_{y_{n}-\zeta}^{y_{n}+\zeta}\rho _{n}(x)dx\geq (s+t)-\varepsilon .
\end{equation*}
\end{itemize}
The method of concentration compactness
, as applied to this situation, consists of the observation
that if $\gamma=s+t,$  then
the minimizing sequence $\{(f_{n},g_{n})\}$ has a subsequence which, up to
translations in the underlying spatial domain, converges strongly in $Y$ to an element of $\mathcal{F}_{s,t}.$ Typically, one proves $\gamma = s+t$ by ruling out
the other two possibilities. We now give the
details of the method, and prove our existence and stability results.

\smallskip

We first establish some properties of $\Theta(s,t)$ and its minimizing
sequences which are independent of the value $\gamma.$

\smallskip

\begin{lem} \label{Ibounded} If $\{(f_n,g_n)\}$ is a minimizing sequence $\Theta(s,t),$ then there exists
constants $B>0$ such that
\begin{equation*}
\|f_n\|_1+\|g_n\|_1 \leq B \ \text{for all} \ n.
\end{equation*}
Moreover, for every $s,t>0,$ one has
$-\infty <\Theta(s,t)<0$.
\end{lem}
\begin{proof}
Using the Gagliardo-Nirenberg inequality,  we have
\begin{equation}
\label{GLforf} |f_{n}|_{p}^{p}\leq
C\|f_{nx}\|^{(p-2)/2} \cdot \|f_{n}\|^{(p+2)/2}.
\end{equation}
Since $ \{(f_{n},g_{n})\}$ is a minimizing sequence, both
$\|f_{n }\|$ and $\|g_{n }\|$ are
bounded. Then, from \eqref{GLforf}, we obtain
\begin{equation}
\label{fLqbound}
 |f_n|_{p}^{p} \le C \|f_{nx}\|^{(p-2)/2}\leq C\|(f_n,g_n)\|_Y^{(p-2)/2},
\end{equation}
where $C$ denotes various constants which are independent of $f_n$ and $g_n.$
Similarly, we have the following estimate
\begin{equation}
\label{gLpbound}
 |g_n|_{r}^{r} \le C \|g_{nx}\|^{(r-2)/2} \le C\|(f_n,g_n)\|_Y^{(r-2)/2}.
\end{equation}
From Cauchy-Schwartz inequality, we also have
\begin{equation}\label{mixtermbound}
 |f_ng_n|_{q}^{q}\ dx
\leq \frac{1}{2}\left(|f_n|_{2q}^{2q}+|g_n|_{2q}^{2q}\right) \le
C\|(f_n,g_n)\|_Y^{q-1}.
\end{equation}

Now, we write
\begin{equation*}
\begin{aligned}
& \|(f_n,g_n)\|_Y^2 =\|f_n\|_1^2+\|g_n\|_1^2 \\
&\ \ =H(f_n,g_n)+
 \left(a|f_{n}|_{p}^{p}+b|g_{n}|_{r}^{r}+c|f_{n}g_n|_{q}^{q}\right) \ dx +(s+t).
\end{aligned}
\end{equation*}
Since $H(f_n,g_n)$ is bounded, we obtain
\begin{equation*}
\|(f_n,g_n)\|_Y^2 \le C\left(1 +
\|(f_n,g_n)\|_Y^{(p-2)/2}+|(f_n,g_n)\|_Y^{(r-2)/2}+\|(f_n,g_n)\|_Y^{q-1}\right),
\end{equation*}
As the norm of the minimizing sequence $ \{(f_{n},g_{n})\}$ is bounded by itself but with smaller power,
the existence of the desired bound $B$ follows.

\smallskip

To see $\Theta(s,t)>-\infty,$ it suffices to bound $H(f,g)$ from below by a number which
is independent of $f$ and $g.$ Using
the estimates \eqref{fLqbound}, \eqref{gLpbound}, and \eqref{mixtermbound}, we obtain
for $(f,g)\in \Sigma_{s,t},$
\begin{equation*}
\begin{aligned}
H(f,g) \geq \|f_x\|^2 & +\|g_x\|^2-C \|f_x\|^{(p-2)/2}-C \|g_x\|^{(r-2)/2}\\
\ & -C (\|f_x\|^{q-1}+\|g_x\|^{q-1}),
\end{aligned}
\end{equation*}
where $C$ denotes various constants independent of $f$ and $g.$ Let us define
\begin{equation*}
Z(x,y)=|x|^2+|y|^2-C (|x|^{(p-2)/2}+|y|^{(r-2)/2}+|x|^{q-1}+|y|^{q-1}).
\end{equation*}
Since $2<p,r,2q<6,$ we have $\varrho :=\min Z(x,y)>-\infty.$
In particular, for all $(f,g)\in \Sigma_{s,t},$ we have that
\begin{equation*}
H(f,g)\geq Z(\|f_x\|,\|g_x\|) \geq \varrho > -\infty.
\end{equation*}

\smallskip

To see that $\Theta(s,t)<0$, choose $(f,g)\in \Sigma_{s,t},$
and $f(x)>0$ and $g(x)>0$ for all $x
\in \mathbb{R}$. For each $\theta >0,$ the functions $f_{\theta
}(x)=\theta ^{1/2}f(\theta x)$ and $g_{\theta }(x)=\theta
^{1/2}g(\theta x)$ satisfy $(f_{\theta},g_{\theta})\in \Sigma_{s,t},$ and
\begin{equation*}
\begin{aligned}
H(f_{\theta },g_{\theta }) &=\int_{-\infty }^{\infty
}\left(|f_{\theta x}|^{2}+|g_{\theta x}|^2
-a|f_{\theta }|^{p}-b|g_{\theta }|^{r}-c|f_{\theta }|^{q}|g_{\theta}|^{q}\right) \ dx \\
&\leq \theta ^{2}\int_{-\infty }^{\infty
}\left(|f_x|^{2}+|g_x|^{2}\right)\ dx -\theta^{q-1}\int_{-\infty
}^{\infty }c|f|^{q}|g|^{q} \ dx.
\end{aligned}
\end{equation*}
Hence, by taking $\theta $ sufficiently small, we get $H(f_{\theta
},g_{\theta})<0.$
\end{proof}

\begin{lem} \label{bdgnxbelow}
Let $(f_n,g_n)$ be a minimizing sequence for $\Theta(s,t)$. Then for all sufficiently large $n,$

(i) if $t>0$ and $s \ge 0$,
then $\exists$ $\delta_{1}>0$ such that $\|g_{nx}\|\geq \delta_{1}.$

(ii) if $s>0$ and $t \geq 0,$ then $\exists$ $\delta_{2}>0$ such that $\|f_{nx}\|\geq \delta_{2}.$
\end{lem}
\begin{proof}
Suppose to the contrary that (i) is false. Then, by passing to a subsequence if necessary, we may assume there exists
a minimizing sequence for which $\displaystyle \lim_{n \to \infty} \|g_{nx}\| = 0$.
By Gagliardo-Nirenberg inequalities, it then follows that
\begin{equation*}
\lim_{n \to \infty} \int_{-\infty}^\infty |g_n|^{r}\ dx = 0 \ \text{and} \ \lim_{n \to \infty} \int_{-\infty }^{\infty }|f_n|^{q}|g_n|^{q}\ dx=0.
\end{equation*}
Therefore, we have that
\begin{equation}
\label{Ibbelow}
\begin{aligned}
\Theta(s,t) &=\lim_{n\to \infty }H(f_n,g_n) \\
&=\lim_{n\to \infty }\int_{-\infty }^{\infty } \left(|f_{nx}|^{2}-a|f_{n}|^{p}\right) \ dx.
\end{aligned}
\end{equation}
Pick any non-negative function $\psi$ such that $\|\psi\|^2=t$.  For every $\theta >0$,
the function $\psi_\theta(x)=\theta^{1/2}\psi(\theta x)$ satisfies $\|\psi_\theta\|^2=t$, and
hence, for all $n,$
\begin{equation*}
\Theta(s,t) \le H(f_n,\psi_\theta).
\end{equation*}
On the other hand, if we define
\begin{equation}\label{elzero}
\eta = \theta^2\int_{-\infty }^{\infty }|\psi_x|^2\ dx-
\theta^{(r-2)/2} \int_{-\infty }^{\infty }b|\psi|^{r}\ dx,
\end{equation}
then $\eta < 0$ for sufficiently small $\theta.$
Then, for all $n\in\mathbb{N}$,
\begin{equation*}
\begin{aligned}
\Theta(s,t) &\leq H(f_n,\psi_\theta) \\
& \le \int_{-\infty}^{\infty } \left(|f_{nx}|^2-a
|f_n|^{p}\right) \ dx + \eta.
\end{aligned}
\end{equation*}
Consequently
\begin{equation*}
\begin{aligned}
\Theta(s,t) &\leq \lim_{n \to \infty}\int_{-\infty}^{\infty
}\left(|f_{nx}|^{2}-a |f_n|^{p}\right)\ dx
+ \eta,
\end{aligned}
\end{equation*}
which contradicts \eqref{Ibbelow} and \eqref{elzero}. The case (ii) can be proved similarly.
\end{proof}
\begin{lem}\label{minforJ}
\label{Is0} Let $1<\alpha<5$ and $\beta > 0$. Define $
J:H^{1}(\mathbb{R}) \to \mathbb R$ by
\begin{equation}\label{defJ}
J(h(x,t))=\int_{-\infty}^\infty\left(|h_x(x,t)|^2 - \beta
|h(x,t)|^{\alpha+1}\right)\ dx.
\end{equation}
Let $s>0$, and let $\{h_n\}$ be any sequence in $H^1$ such that
$\|h_n\|^2 \to s$
 and
\begin{equation*}
\lim_{n \to \infty} J(h_n) = \inf \ \left\{J(h): h \in H^1\ {\rm
and }\ \|h\|^2 = s\right\}.
\end{equation*}
Then there exists a subsequence $\{h_{n_k}\}$, a family $\{y_k\}\subset \mathbb{R},$
and a real number $\theta$ such that $e^{-i\theta}h_{n_k}(x+y_k)$ converges strongly in
$H^1$ norm to $h_s(x)$, where
\begin{equation}
\label{h0}
 h_s(x)=\left(\frac{\lambda}{\beta}\right)^{1/(\alpha-1)}{\rm
sech}^{2/(\alpha-1)}\left(\frac{\sqrt{\lambda} (\alpha-1) x}{2}\right),
\end{equation}
and $\lambda > 0$ is chosen so that $\|h_s\|^2 = s$. In particular,
\begin{equation}
J(h_s)=\inf \ \left\{J(h): h \in H^1\
{\rm and }\ \|h\|^2 = s\right\}. \label{f0ismin}
\end{equation}
\end{lem}
\begin{proof}
The fact that some translated subsequence of $h_{n}$ must converge
strongly in $H^{1}$ norm can be proved by
the use of Cazenave-Lions method (see, for example, \cite{[CL], [Ca]}).

\smallskip

Let $\varphi \in H^1$ be the limit of the translated subsequence $\{h_{n_k}(x+\tilde y_k)\}$ of $\{h_n\}$.
Then the limit function $\varphi$ satisfies
\begin{equation}
\label{E220}
J(\varphi)=\inf \ \left\{J(h): h \in H^1\ {\rm and }\ \|h\|^2 =
s\right\},
\end{equation}
and also be a solution of
\begin{equation}
-2\varphi'' -(\alpha+1)\beta \varphi^{\alpha} = -2\lambda \varphi \label{elforg}
\end{equation}
for some real number $\lambda.$
It is well known (see Theorem 8.1.6 of \cite{[Ca]}) that the solutions of \eqref{elforg} can be described explicitly by
\begin{equation*}
\{e^{i\theta} h_s(\cdot + y_s),\ y_s, \theta \in \mathbb{R}\}.
\end{equation*}
Then \eqref{f0ismin} follows from \eqref{E220}. Also, if we define $y_k =\tilde y_k-y_s$, then we have that
$e^{-i\theta}h_{n_k}(x+y_k)$ converges in $H^1$ to $h_s$.
\end{proof}

\begin{lem} \label{negative}
Suppose $(f_n,g_n)$ is a minimizing sequence for $\Theta(s,t)$, where
$s>0$ and $t >0$.  Then there exists $\delta_{1}>0$ and $\delta_{2}>0$ such that for all
sufficiently large $n$,
\begin{equation*}
|f_{nx}|_{2}^2 -a|f_n|_{p}^{p}
-c|f_ng_n|_{q}^q \le -\delta_{1},\ \text{and} \ |g_{nx}|_{2}^2 -b|g_n|_{r}^{r}
-c|f_ng_n|_{q}^q \le -\delta_{2}.
\end{equation*}
\end{lem}
\begin{proof} Both inequalities can be proved by using similar arguments. We only prove the first inequality.
Suppose the conclusion is false. Then, by passing to a
subsequence if necessary, we may assume that there exists a minimizing sequence
$(f_n,g_n)$ for which
\begin{equation}
\liminf_{n \to \infty} \left(|f_{nx}|_{2}^2 -a|f_n|_{p}^{p}
-c|f_ng_n|_{q}^q\right) \ge 0,
\end{equation}
and so
\begin{equation}
\Theta(s,t) = \lim_{n\to\infty}H(f_n,g_n) \ge \liminf_{n \to \infty}
\int_{-\infty}^\infty\left(|g_{nx}|^2 - b |g_n|^{r}\right)\
dx. \label{IgeJg0}
\end{equation}

Define $J$ and $g_t$ as in Lemma \ref{minforJ} with $h=g, \beta=b,$ and $\alpha=r-1.$  Then
\eqref{IgeJg0} implies that
\begin{equation}
\Theta(s,t) \ge J(g_t). \label{IgeJg02}
\end{equation} On the other hand, take any $f \in H^1$ such that $\|f\|^2=s$ and
\begin{equation}\label{minf2}
|f_{x}|_{2}^2 -a|f|_{p}^{p}
-c|fg_t|_{q}^q < 0.
\end{equation}
To construct such a function $f,$ take an arbitrary smooth, non-negative function $\psi$ with compact support
such that $\psi(0)=1$ and $\|\psi\|=s,$ and for $\theta > 0,$ define $\psi_{\theta}(x)=\theta^{1/2}\psi(\theta x).$
Then, $f=\psi_{\theta}$ satisfies \eqref{minf2} for sufficiently small $\theta.$
Therefore,
\begin{equation}
\Theta(s,t) \le H(f,g_t) \le |f_{x}|_{2}^2 -a|f|_{p}^{p}
-c|fg_t|_{q}^q +J(g_t) < J(g_t),
\end{equation}
which contradicts \eqref{IgeJg02}, and hence lemma follows.
\end{proof}

\begin{lem} \label{posmin}
$H(|f|,|g|)\leq H(f,g)$ for all $(f,g)\in Y$.
\end{lem}
\begin{proof}
The proof follows from the fact that if $f\in H^{1},$ then $%
\left\vert f(x)\right\vert $ is in $H^{1}$ and%
\begin{equation}\label{decrea2}
\int_{-\infty }^{\infty }\left\vert \left\vert f\right\vert _{x}\right\vert
^{2}\ dx\leq \int_{-\infty }^{\infty }\left\vert f_{x}\right\vert ^{2}\ dx.
\end{equation}
A proof of \eqref{decrea2} can be given by working with Fourier transforms of $f$ and $|f|$ and
is easily constructed by adapting the proof of Lemma~3.5 in \cite{[ABS]}.
\end{proof}

In the sequel, we denote by $e^\ast(x)$ the symmetric decreasing rearrangement for a function
$e:\mathbb{R} \to [0,\infty).$ We refer the reader to \cite{[LL]} for details
about symmetric decreasing rearrangements. We note here that if $(f,g)\in Y,$ then $|f|, |g| \in Y,$ and hence
symmetric rearrangements $|f|^\ast$ and $|g|^\ast$ of $|f|$ and $|g|$ are well-defined.
A basic property about symmetric decreasing rearrangement is that $L^p$ norms are preserved:
\begin{equation}\label{Lpnormpreserve}
\int_{-\infty }^{\infty }(|f|^\ast)^{p}\ dx = \int_{-\infty }^{\infty }|f|^{p}\ dx
\end{equation}

\begin{lem}\label{symmmin}
$H(|f|^\ast,|g|^\ast)\leq H(f,g)$ for all $(f,g)\in Y.$
\end{lem}
\begin{proof}
From Theorem 3.4 of \cite{[LL]}, we have
\begin{equation}\label{Lppreserve2}
\int_{-\infty }^{\infty }(|f|^\ast) ^{q}(|g|^\ast)^{q} \ dx \geq
\int_{-\infty }^{\infty }|f|^{q}|g|^{q}\ dx.
\end{equation}
Lemma 7.17 of \cite{[LL]} implies that
\begin{equation}\label{Lppreserve3}
\int_{-\infty }^{\infty }| (|f|^\ast)_{x}| ^{2}\ dx\leq
\int_{-\infty }^{\infty }||f|_{x}|^{2}\ dx,
\end{equation}
and similarly for $g(x)$. Then, the claim follows by using the
facts \eqref{Lpnormpreserve}, \eqref{Lppreserve2}, \eqref{Lppreserve3}, and
Lemma \ref{posmin}.
\end{proof}

The next lemma is one-dimensional version of Proposition 1.4 of \cite{[By]}. A proof of this lemma is given in \cite{[AB11]} (see also \cite{[Gar]}).

\begin{lem}  Suppose $u, v:\mathbb{R}\to [0,\infty)$ are even, $C^\infty,$
non-increasing, and have compact support in $\mathbb R.$  Let $a_1$ and $a_2$ be real numbers such that
$\text{supp}(u(x+a_1))\cap \text{supp}(v(x+a_2))=\emptyset,$
and define
$$
e(x)=u(x+a_1)+v(x+a_2).
$$
Then the derivative $(e^\ast)'$ of
$e^\ast$ (in sense of distribution) is in $L^2$, and satisfies
\begin{equation}
\|(e^\ast)'\|^2 \le \|e'\|^2 - \frac34 \min\{\|u'\|^2,\|v'\|^2\}.
\label{garineq}
\end{equation}
\label{garlem}
\end{lem}

\smallskip

The next lemma proves that $\Theta(s,t)$ is subadditive:

\begin{lem} \label{subadd}
Let $s_1,s_2,t_1,t_2\geq 0$ be such that
$s_1+s_2>0$, $t_1+t_2>0$, $s_1+t_1>0$, and $s_2+t_2>0$. Then
\begin{equation}
\label{SUBA}
\Theta(s_{1}+s_{2},t_{1}+t_{2})<\Theta(s_{1},t_{1})+\Theta(s_{2},t_{2}).
\end{equation}
\end{lem}
\begin{proof}  Let $i=1,2.$ Then, following closely the arguments used in \cite{[AB11]},
we can choose
minimizing sequences
$(f_n^{(i)},g_n^{(i)})$ for $\Theta(s_i,t_i)$ such that $f_n^{(i)}$ and $g_n^{(i)}$
are real-valued, non-negative, even, $C^\infty$ with compact support in $\mathbb{R},$ non-increasing on $\{x:x\ge 0\},$ and satisfy
$(f_n^{(i)}, g_n^{(i)})\in \Sigma_{s_i,t_i}.$

\smallskip

Now, for each each $n,$ choose a number $x_n$
such that $f_n^{(1)}(x)$ and $\tilde f_n^{(2)}(x)= f_n^{(2)}(x+x_n)$ have disjoint support, and
$g_n^{(1)}(x)$ and $\tilde g_n^{(2)}(x)=g_n^{(2)}(x+x_n)$ have disjoint support.  Define
\begin{equation*}
\begin{aligned}
f_n &= \left(f_n^{(1)} + \tilde f_n^{(2)}\right)^\ast \ \text{and} \
g_n &= \left(g_n^{(1)} + \tilde g_n^{(2)}\right)^\ast.
\end{aligned}
\end{equation*}
Then $(f_n,g_n) \in \Sigma_{s_1+s_2,t_1+t_2},$ and hence,
\begin{equation}
\Theta(s_1+s_2,t_1 + t_2) \le H(f_n,g_n).
\label{IleE}
\end{equation}
On the other hand, from Lemma \ref{garlem} we have that
\begin{equation}
\begin{aligned}
\int_{-\infty}^\infty &{\left( f_{nx}^2 + g_{nx}^2 \right)\ dx}
\le
\int_{-\infty}^\infty {\left( (f_n^{(1)}+\tilde f_n^{(2)})_x^2 + (g_n^{(1)}+\tilde g_n^{(2)})_x^2 \right)\ dx} - K_n\\
 &=\int_{-\infty}^\infty {\left( (f_{nx}^{(1)})^2+(\tilde f_{nx}^{(2)})^2 + (g_{nx}^{(1)})^2+(\tilde g_{nx}^{(2)})^2 \right)\ dx} - K_n,
\end{aligned}
\label{kinenstrictdec}
\end{equation}
where
\begin{equation}\label{defKn}
K_n = \frac34
\left(\min\left\{\|f_{nx}^{(1)}\|^2,\|f_{nx}^{(2)}\|^2\right\}
+\min\left\{\|g_{nx}^{(1)}\|^2,\|g_{nx}^{(2)}\|^2\right\}\right).
\end{equation}
Moreover, from the properties of rearrangements, we have that
\begin{equation} \label{rear1}
\begin{aligned}
\int_{-\infty}^\infty |f_{n}|^{p}\ dx &=\int_{-\infty}^\infty |f_{n}^{(1)}|^{p}\ dx+\int_{-\infty}^\infty |f_{n}^{(2)}|^{p}\ dx ,\\
\int_{-\infty}^\infty |g_{n}|^{r}\ dx &=\int_{-\infty}^\infty |g_{n}^{(1)}|^{r}\ dx+\int_{-\infty}^\infty |g_{n}^{(2)}|^{r}\ dx ,\\
\int_{-\infty}^\infty |f_n|^q|g_n|^{q}\ dx &\ge \int_{-\infty}^\infty |f_n^{(1)}|^q|g_n^{(1)}|^{q}\ dx + \int_{-\infty}^\infty |f_n^{(2)}|^q | g_n^{(2)}|^q\ dx.
\end{aligned}
\end{equation}
Then, \eqref{IleE}, \eqref{kinenstrictdec} and \eqref{rear1} give, for all $n,$
\begin{equation}
\Theta(s_1+t_1, s_2+t_2) \le H(f_n,g_n) \le H(f_n^{(1)},g_n^{(1)}) + H(f_n^{(2)}, g_n^{(2)}) - K_n.
\end{equation}
Hence, we obtain
\begin{equation}
\Theta(s_1+t_1,s_2+t_2) \le \Theta(s_1,t_1)+\Theta(s_2,t_2) - \liminf_{n \to \infty} K_n.
\label{subaddKn}
\end{equation}
Since $t_1 + t_2 >0$, we
consider the following three cases: (i) $t_1
> 0$ and $t_2>0$; (ii) $t_1 = 0$, $t_2 > 0$, and $s_2 > 0$; and (iii)
$t_1 = 0$, $t_2 >0$, and $s_2=0$.

\smallskip

\textit{Case 1:} When $t_1 > 0$ and $t_2>0.$
Lemma
\ref{bdgnxbelow} guarantees that there exist numbers $\delta_1 >0$ and
$\delta_2 >0$ such that for all sufficiently large $n$,
\begin{equation*}
\|(g_n^{(1)})_x\| \ge \delta_1 \ \text{and}\ \|(g_n^{(2)})_x\| \ge
\delta_2.
\end{equation*}
Let $\delta =
\min(\delta_1, \delta_2)> 0$. Then, \eqref{defKn} gives $K_n \ge
3\delta/4$ for all sufficiently large $n$. From \eqref{subaddKn}
we have
\begin{equation*}
\Theta(s_1+t_1,s_2+t_2) \le \Theta(s_1,t_1)+\Theta(s_2,t_2) - 3\delta/4 <
\Theta(s_1,t_1)+\Theta(s_2,t_2), \label{subaddwdelta}
\end{equation*} as
desired.

\smallskip

\textit{Case 2:} When $t_1 = 0, t_{2} >0,$ and $s_2>0.$ Since $s_{1}+t_{1}>0, s_{1}>0$ too.
By Lemma
\ref{bdgnxbelow}, there exist numbers $\delta_3 >0$ and
$\delta_4 >0$ such that for all sufficiently large $n$,
\begin{equation*}
\|(f_n^{(1)})_x\| \ge \delta_3 \ \text{and}\ \|(f_n^{(2)})_x\| \ge
\delta_4.
\end{equation*}
Let $\delta =
\min(\delta_3, \delta_4)> 0$. Then, \eqref{defKn} gives $K_n \ge
3\delta/4$ for all sufficiently large $n$. From \eqref{subaddKn}
we have
\begin{equation*}
\Theta(s_1+t_1,s_2+t_2) \le \Theta(s_1,t_1)+\Theta(s_2,t_2) - 3\delta/4 <
\Theta(s_1,t_1)+\Theta(s_2,t_2). \label{subaddwdelta}
\end{equation*}

\smallskip

\textit{Case 3:} When $t_1 = 0, t_{2} >0,$ and $s_2=0.$ In this case, we have
\begin{equation*}
\Theta(0,t_{2})=\inf \left\{\int_{-\infty }^{\infty
}\left(|g_{x}|^{2}-b|g|^{r}\right)\ dx: g\in H^{1} \ \text{and}\ \|g\|^{2}=t_{2}>0 \right\}
\end{equation*}
and
\begin{equation*}
\Theta(s_{1},0)=\inf \left\{\int_{-\infty }^{\infty
}\left(|f_{x}|^{2}-a|f|^{p}\right)\ dx: f\in H^{1} \ \text{and}\ \|f\|^{2}=s_{1}>0 \right\}.
\end{equation*}
Lemma~\ref{minforJ} with $h=g, s=t_{2}, \beta=b,$ and $\alpha=r-1$ implies $\Theta(0,t_{2})=J(g_{t_2}).$ Similarly, let $f_{s_1}$
be such that $\Theta(s_{1},0)=J(f_{s_1}).$ Clearly,
\begin{equation*}
\int_{-\infty }^{\infty}|f_{s_1}|^q |g_{t_2}|^q\ dx >0
\end{equation*}
and so
\begin{equation*}
\begin{aligned}
\Theta(s_{1},t_{2})\leq H(f_{s_{1}},g_{t_{2}})&=\Theta(s_{1},0)+\Theta(0,t_{2})-c\int_{-\infty}^{\infty}|f_{s_{1}}|^{q} |g_{t_{2}}|^{q}\ dx \\
&<\Theta(s_{1},0)+\Theta(0,t_{2}).
\end{aligned}
\end{equation*}
This completes the proof of lemma.
\end{proof}

\begin{lem}\label{GAst}
Suppose $\gamma =s+t$ and let $\{(f_n,g_n)\}$ be a minimizing sequence for $\Theta(s,t).$ Then there exists a sequence of real numbers $%
\{y_{n}\}$ such that

\smallskip

$1.$ for every $z<s+t$ there exists $\zeta =\zeta(z)$ such that
\begin{equation*}
\int_{y_{n}-\zeta }^{y_{n}+\zeta }(|f_n|^{2}+|g_n|^{2})\ dx > z
\end{equation*}
for all sufficiently large $n.$

\smallskip

$2.$ the sequence $\{(w_n,z_n)\}$ defined by
\begin{equation*}
w_n(x)=f_n(x+y_n)\ \text{and}\ z_n(x)=g_n(x+y_n),\ x\in \mathbb{R},
\end{equation*}
has a subsequence which converges in $Y$ norm to a function $(\Phi,\Psi) \in \mathcal{F}_{s,t}.$
In particular, $\mathcal{F}_{s,t}$ is nonempty.
\end{lem}
\begin{proof}
Since $\gamma =s+t,$ then, by the definition of $\gamma ,$ there exists $%
\zeta_{0}$ such that for $n$ sufficiently large, $M_n(\zeta_0)>(s+t)/2.$
Thus, for each sufficiently large $n,$ we can find $y_{n}$ such that
\begin{equation*}
\int_{y_{n}-\zeta_{0}}^{y_{n}+\zeta
_{0}}(|f_{n}|^{2}+|g_n|^{2}) \ dx > \frac{s+t}{2}.
\end{equation*}
Now, let $z<s+t.$ Clearly, we may assume $z \in (\frac{s+t}{2}, s+t).$ Again, since $%
\gamma =s+t,$ we can find $\zeta_{1}=\zeta_1(z),$ such that for $n$ sufficiently large, $M_n(\zeta_1)>z,$ and so, we can
choose $\tilde{y}_n$ such that
\begin{equation*}
\int_{\tilde{y}_n-\zeta_{1}}^{\tilde{y}_n+\zeta_{1}
}(|f_{n}|^{2}+|g_n|^{2}) \ dx > z
\end{equation*}
for some $\tilde{y}_n \in \mathbb{R}.$ Since
$\int_{-\infty }^{\infty }(|f_{n}|^{2}+|g_n|^{2})\
dx=s+t,$
it follows that for large $n,$ the intervals $[\tilde {y}_{n}-\zeta_{1},\tilde{y}_{n}+\zeta
_{1}]$ and $[y_{n}-\zeta_{0},y_{n}+\zeta _{0}]$ must overlap.
Then, by defining $\zeta =2 \zeta_{1}+\zeta _{0},$ we have that $%
[y_{n}-\zeta ,y_{n}+\zeta ]$ contains $[\tilde{y}_{n}-\zeta _{1},\tilde{y}_{n}+\zeta
_{1}],$ and the statement $1$ follows.

\smallskip

To prove statement 2, notice first that statement 1 implies that,
for every $k\in \mathbb{N},$ there exists $\zeta _{k}\in \mathbb{R}$ such
that
\begin{equation} \label{EJ10}
\int_{-\zeta _{k}}^{\zeta _{k}}\left(  |w_{n}|
^{2}+|z_{n}|^{2}\right) \ dx>s+t-\frac{1}{k},
\end{equation}
for all sufficiently large $n.$
Since $\{(w_{n},z_{n})\}$ is bounded
uniformly in $Y,$ there exists a subsequence, denoted again by $%
\{(w_{n},z_{n})\},$ which converges weakly in $Y$ to a limit $(\Phi,\Psi)\in Y.$
Then Fatou's lemma implies that%
\begin{equation*}
\| \Phi\|^{2}+\left\Vert \Psi \right\Vert ^{2}\leq \ \underset{%
n\rightarrow \infty }{\lim \inf }\left(
\|w_{n}\|^{2}+\|z_{n}\|^{2}\right)=s+t.
\end{equation*}
Moreover, for fixed $k,\ (w_{n},z_{n})$ converges weakly in $H^{1}(-\zeta
_{k},\zeta _{k})\times H^{1}(-\zeta _{k},\zeta _{k})$ to $(\Phi,\Psi),$ and
therefore has a subsequence, denoted again by $\{(w_{n},z_{n})\},$ which
converges strongly to $(\Phi,\Psi)$ in $L^{2}(-\zeta _{k},\zeta _{k})\times
L^{2}(-\zeta _{k},\zeta _{k}).$ By a diagonalization argument, we may
assume that the subsequence has this property for every $k$ simultaneously.
It then follows from \eqref{EJ10} that%
\begin{equation*}
\| \Phi\|^{2}+\left\Vert \Psi \right\Vert ^{2} \geq \int_{-\zeta _{k}}^{\zeta _{k}}\left(|\Phi|
^{2}+|\Psi|^{2}\right) \ dx\geq s+t-\frac{1}{k}.
\end{equation*}
Since $k$ was arbitrary, we get%
\begin{equation*}
\| \Phi\|^{2}+\left\Vert \Psi \right\Vert ^{2}=s+t,
\end{equation*}
which implies that $(w_{n},z_{n})$ converges strongly to the limit $(\Phi,\Psi)$ in $%
L^{2}\times L^{2}.$

\smallskip

Next, observe that
\begin{equation*}
|z_{n}-\Psi|_{r}^{r}\leq C\|z_{n}-\Psi\|
_{1}^{1/r}\|z_{n}-\Psi\|^{(r-1)/r}\leq
C\|z_{n}-\Psi|^{(r-1)/r},
\end{equation*}
which implies $|z_{n}|_{r}^{r}\to
|\Psi| _{r}^{r}$ as $n\to \infty .$ Also,%
\begin{equation*}
|w_{n}-\Phi|_{p}^{p}\leq C\left\Vert w_{n}-\Phi\right\Vert
_{1}^{1/p}\left\Vert w_{n}-\Phi\right\Vert ^{(p-1)/p}\leq
C\left\Vert w_{n}-\Phi\right\Vert ^{(p-1)/p},
\end{equation*}
and hence $\left\vert w_{n}\right\vert _{p}^{p}\to \left\vert
\Phi\right\vert _{p}^{p}$ as $n\to \infty .$ The fact
\begin{equation*}
\lim_{n\to \infty }\int_{-\infty }^{\infty }|z_{n}|^{q}|
w_{n}|^{q}\ dx=\int_{-\infty }^{\infty }|\Psi|^{q}|\Phi|
^{q}\ dx
\end{equation*}
follows by writing
\begin{equation*}
\begin{aligned}
\int_{-\infty }^{\infty }(|z_{n}|^{q}|w_{n}|^{q}-|\Psi|^{q}|
\Phi|^{q})\ dx &= \int_{-\infty }^{\infty }|z_{n}|^{q}\left(
|w_{n}|^{q}-|\Phi|^{q}\right) \ dx \\
&+\int_{-\infty }^{\infty }(|z_{n}|^{q}-|\Psi|^{q})|\Phi|^{q}\ dx
\end{aligned}
\end{equation*}
and noting that $\{(w _{n},z_{n})\}$ is bounded in $Y.$
Therefore, by another
application of Fatou's lemma, we get%
\begin{equation}\label{tech1}
\Theta(s,t)=\lim_{n \to \infty} H(w_n,z_n) \ge H(\Phi,\Psi);
\end{equation}
whence $H(f,g)=\Theta(s,t).$ Thus $(\Phi,\Psi)\in \mathcal{F}_{s,t}.$
Finally, since equality holds in \eqref{tech1},
one has
\begin{equation*}
\lim_{n\to \infty} \left(\|w_{nx}\|
^2+\|z_{nx}\|^2\right) =
\|\Phi_x\|^2+\|\Psi_x\|^2,
\end{equation*}
so $(w_n(x),z_n(x))$ converges strongly to $(\Phi,\Psi)$
in the norm of $Y$.
\end{proof}

\smallskip

The following result, which we state here without proof, is a special case of Lemma I.1 of \cite{[L]}. For a proof, see Lemma 2.13 of \cite{[AB11]}.
\begin{lem}  \label{Lionsvanish}
Suppose $f_n$ is a bounded sequence in $H^1(\mathbb{R})$
such that, for some $R >0$,
\begin{equation}
\lim_{n \to \infty} \sup_{y \in \mathbb R} \int_{y-R}^{y+R} f_n^2\ dx = 0.
\label{vanishhypo}
\end{equation}
Then for every $k>2$,
\begin{equation*}
\lim_{n \to \infty} |f_n|_k = 0.
\end{equation*}
\end{lem}
We can now rule out the case of vanishing:
\begin{lem}\label{GAvanish}
For any minimizing sequence $\{(f_{n},g_{n})\}\in Y,\ \gamma > 0.$
\end{lem}
\begin{proof}
Suppose to contrary that $\gamma =0$.
By Lemma~\ref{Ibounded}, both $\{|f_n| \}$ and $\{|g_n| \}$ are bounded
sequences in $H^1$. Using Lemma
\eqref{Lionsvanish}, for every $k
> 2$, $f_n$ and $g_n$ converge to 0 in $L^k$ norm. In particular,
$|f_n|_{p}^{p} \to 0$ and
$|g_n|_{r}^{r} \to 0.$
 Since
\begin{equation*}
\int_{-\infty }^{\infty }|f_n|^{q}|g_n|^{q}\ dx \leq
|f_n|_{2q}^{q}|g_n|_{2q}^{q},
\end{equation*}
it follows also that
\begin{equation*}
\lim_{n \to \infty} \int_{-\infty }^{\infty }|f_n|^{q}|g_n|^{q}\ dx =0.
\end{equation*}
Hence
\begin{equation}
\Theta(s,t)=\lim_{n\to \infty }H(f_n,g_n) \ge \liminf_{n \to
\infty}\int_{-\infty }^{\infty }\left( |f_{nx}|^2+|g_{nx}|^2\right)
dx\geq 0,
\label{istgezero}
\end{equation}
contradicting Lemma~\ref{Ibounded}. This proves $\gamma > 0.$
\end{proof}

\begin{lem} \label{DIC} There exist $s_1 \in [0,s]$ and $t_1 \in
[0,t]$ such that
\begin{equation}
\label{SUM} \gamma = s_1+t_1
\end{equation}
and
\begin{equation} \label{REV}
\Theta(s_{1},t_{1})+\Theta(s-s_{1},t-t_{1})\leq \Theta(s,t).
\end{equation}
\end{lem}
\begin{proof}
Let $\epsilon$ be an arbitrary positive number. From the definition of $\gamma,$ it follows
that for $\zeta$ sufficiently large, we have $\gamma-\epsilon<M(\zeta)\leq M(2\zeta) \leq \gamma.$ By taking $\zeta$ larger if necessary, we may also assume that $\frac{1}{\zeta}<\epsilon.$
From the definition of $M,$ we can choose $N$ so large that, for every $n\geq N,$
\begin{equation*}
\gamma-\epsilon < M_{n}(\zeta) \leq M_{n}(2\zeta) \leq \gamma +\epsilon.
\end{equation*}
Hence, for each $n\geq N,$ we can find $y_n$ such that
\begin{equation}\label{sub113}
\int_{y_{n-\zeta}}^{y_{n}+\zeta}(|f_{n}|
^{2}+|g_{n}|^{2})\ dx > \gamma -\epsilon \ \text{and} \ %
\int_{y_{n-2\zeta}}^{y_{n}+2\zeta}(|f_{n}|
^{2}+|g_{n}|^{2})\ dx < \gamma +\epsilon.
\end{equation}
Now choose smooth functions $\rho$ and
$\sigma$ on $\mathbb R$ such that $\rho^2 + \sigma^2 = 1$ on $\mathbb R$, and $\rho$ is identically 1
on $[-1,1]$ and has support in $[-2,2]$. Set, for $\zeta > 0,$
\begin{equation*}
\rho_\zeta(x)=\rho(x/\zeta)\ \text{and} \ \sigma_\zeta(x)=\sigma(x/\zeta).
\end{equation*}
From the definition of $\gamma,$ it follows that for given $\epsilon >0$, there exist $\zeta > 0$
and a sequence $y_n$ such that, after passing to a subsequence, the functions defined by
\begin{equation*}
(f_n^{(1)}(x), g_n^{(1)}(x))=\rho_\zeta(x-y_n)(f_n(x),g_n(x))
\end{equation*}
and
\begin{equation*}
(f_n^{(2)}(x), g_n^{(2)}(x))=\sigma_\zeta(x-y_n)(f_n(x),g_n(x))
\end{equation*}
satisfy
\begin{equation*}
\|f_n^{(1)}\|^2 \to s_1, \ \|g_n^{(1)}\|^2 \to t_1, \  \|f_n^{(2)}\|^2 \to s-s_1, \ \text{and}\ \|g_n^{(2)}\|^2 \to t-t_1,
\end{equation*}
as $n \to \infty$. Now
\begin{equation*}
s_{1}+t_{1}=\lim_{n\to \infty }\int_{-\infty
}^{\infty }(|f_{n}^{(1)}|^{2}+|
g_{n}^{(1)}|^{2})\ dx=\lim_{n\to \infty
}\int_{-\infty }^{\infty }\rho _{\zeta }(|
f_{n}|^{2}+|g_{n}|^{2})\ dx.
\end{equation*}
From \eqref{sub113}, it follows that, for every $n\in \mathbb{N},$
\begin{equation*}
\gamma -\epsilon <\int_{-\infty }^{\infty }\rho _{\zeta
}(|f_{n}|^{2}+|g_{n}|^{2})\
dx<\gamma +\epsilon .
\end{equation*}
Hence $|(s_{1}+t_{1})-\gamma| <\epsilon .$
We claim that for all $n,$
\begin{equation} \label{e12ineq}
H(f_n^{(1)},g_n^{(1)})+H(f_n^{(2)},g_n^{(2)}) \le H(f_n,g_n) + C\epsilon
\end{equation}
To see \eqref{e12ineq}, we write
\begin{equation*}
\begin{aligned}
&H(f_n^{(1)},g_n^{(1)})= \int_{-\infty }^{\infty }\rho _{\zeta }^{2}\left(
(|f_{nx}|^{2}+|g_{nx}|^{2})-(a|f_{n}|^{p}+b|g_{n}|^{r}+c|f_{n}g_{n}|^{q})\right)dx\\
&\ \ +\int_{-\infty }^{\infty }\left(a(\rho _{\zeta
}^{2}-\rho _{\zeta }^{p})|f_{n}|^{p}+b(\rho _{\zeta
}^{2}-\rho _{\zeta }^{r})|g_{n}|^{r}+c(\rho _{\zeta
}^{2}-\rho _{\zeta }^{2q})|f_{n}|^{q}|g_{n}|^{q}\right)\ dx\\
&\ \  +\int_{-\infty }^{\infty }\left(\left( \rho _{\zeta }^{\prime
}\right) ^{2}\left( |f_{n}|^{2}+|g_{n}|^{2}\right) +2\rho _{\zeta }^{\prime
}\rho _{\zeta }\left(\text{Re}(f_{n}(\bar{f_{n}})_{x}+\text{Re}(g_{n}(\bar{g_{n}})_{x}\right)\right) \ dx.
\end{aligned}
\end{equation*}
and observe that the last two integrals on the right hand side can be made
arbitrarily uniformly small by taking $\zeta$ sufficiently large.
Similarly, we can estimate for $H(f_n^{(2)},g_n^{(2)}).$
Then, \eqref{e12ineq} follows by adding these two estimates, because $\rho_{\zeta}^2 + \sigma_{\zeta}^2 = 1.$

Now, if $s_{1},t_{1},s-s_{1},$ and $t-t_{1}$ are all positive, then the
claim follows by re-scaling $f_{n}^{(i)}$ and $g_{n}^{(i)}\ (i=1,2).$ Indeed,
let
\begin{equation*}
\alpha _{n}=\frac{\sqrt{s_{1}}}{\| f_{n}^{(1)}\| },\ \beta
_{n}=\frac{\sqrt{t_{1}}}{\| g_{n}^{(1)}\| },\ \gamma _{n}=%
\frac{\sqrt{s-s_{1}}}{\| f_{n}^{(2)}\| },\ \theta _{n}=%
\frac{\sqrt{t-t_{1}}}{\| g_{n}^{(2)}\| },
\end{equation*}
which gives
\begin{equation*}
(\alpha _{n}f_{n}^{(1)},\beta_{n}g_{n}^{(1)})\in \Sigma_{s_1,t_1}\ \text{and} \
(\gamma_{n}f_{n}^{(2)},\theta_{n}g_{n}^{(2)}) \in \Sigma_{s-s_1,t-t_1}.
\end{equation*}
As all the scaling factors tend to $1$ as $n\to \infty ,$
\begin{equation*}
\liminf_{n\to \infty
}\left(H(f_{n}^{(1)},g_{n}^{(1)})+H(f_{n}^{(2)},g_{n}^{(2)})\right)\geq
\Theta(s_{1},t_{1})+\Theta(s-s_{1},t-t_{1}).
\end{equation*}
If $s_1=0$ and $t_1>0,$ then we have
\begin{equation*}
\begin{aligned}
\lim_{n \to \infty}H(f_n^{(1)},g_n^{(1)})&=\lim_{n \to
\infty}\int_{-\infty}^\infty\left(|(f_{n}^{(1)})_{x}|^{2}+|(g_{n}^{(1)})_{x}|^{2}-b|g_{n}^{(1)}|^{r}
\right)\ dx\\
&\ge \liminf_{n\to \infty}\int_{-\infty}^\infty \left(|(g_{n}^{(1)})_{x}|^{2}-b|g_{n}^{(1)}|^{r}\right)
\ dx \ge \Theta(0,t_1).
\end{aligned}
\end{equation*}
Similar estimates hold if $t_1$, $s-s_1$, or $t-t_1$ are zero. Thus, in all the cases we have that the limit inferior as $n\to \infty$ of the left hand side of \eqref{e12ineq} $\geq \Theta(s_{1},t_{1})+\Theta(s-s_{1},t-t_{1}).$ Consequently,
\begin{equation*}
\Theta(s_{1},t_{1})+\Theta(s-s_{1},t-t_{1})\leq \Theta(s,t) + C\epsilon,
\end{equation*}
which proves the lemma, as $\epsilon$ is arbitrary.

\end{proof}

The following lemma rules out the possibility of dichotomy of minimizing sequences:
\begin{lem}\label{GAdicho}
For every minimizing sequence, one has $\gamma \not\in (0,s+t).$
\end{lem}
\begin{proof}
Suppose to the contrary that $\gamma$ satisfies $0<\gamma < s + t$.   Let $s_1$ and
$t_1$ be as in Lemma~\ref{DIC}, and let $s_2=s-s_1$ and
$t_2=t-t_1$. Then $s_{2}+t_{2}=(s+t)-\gamma >0$, and also
$s_{1}+t_{1}=\gamma >0$.
Furthermore, $s_{1}+s_{2}=s>0$ and $t_{1}+t_{2}=t>0$. Therefore
Lemma~\ref{subadd} implies that that \eqref{SUBA} holds. But this
contradicts \eqref{REV} and thus, lemma follows.
\end{proof}

Thus all the preliminaries for the proofs of Theorem~\ref{existence} have been established.
We are now able to prove statements (a)-(d) of Theorem~\ref{existence}.
\begin{proof}[Proof of Theorem \ref{existence}]
 From Lemmas~\ref{GAvanish} and \ref{GAdicho}, it follows that every minimizing sequence must be compact, i.e., $\gamma=s+t.$ Then the
 statement (a) of the Theorem \ref{existence} follows from Lemma~\ref{GAst}.

 \smallskip

To see the validity of statement (b), notice that $(\Phi,\Psi)$ is in the minimizing set $\mathcal F_{s,t}$
for $\Theta(s,t)$, and so minimizes $H(u,v)$ subject to $Q(u)$ and
$Q(v)$ being held constant, the Lagrange multiplier principle
asserts that
there exist real numbers $\omega_{1}$ and $\omega_{2}$ such that
\begin{equation}
\delta H(\Phi,\Psi) + \omega_{1} \delta Q(\Phi)+\omega_{2} \delta Q(\Psi)=0,
\end{equation}
where $\delta$ denotes the Fr\'echet derivative.  Computing the associated
Fr\'echet derivatives we see that the equations
\begin{equation}\label{ODE1}
\left\{
\begin{aligned}
-\Phi ^{\prime \prime }+\omega_{1}\Phi &=\alpha|\Phi|^{p-2}\Phi+\tau|\Psi|^{q}|\Phi|^{q-2}\Phi, \\
-\Psi ^{\prime \prime }+\omega_{2}\Psi &=\beta|\Psi|^{r-2}\Psi+\tau|\Phi|^{q}|\Psi|^{q-2}\Psi,
\end{aligned}
\right.
\end{equation}
hold, at least in the sense of distributions. A straightforward bootstrapping argument (cf.\ Lemma 1.3 of \cite{[T]}) shows that
distributional
solutions are also classical solutions.

\smallskip

Multiplying the first equation in \eqref{ODE1} by $\bar{\Phi }$ and the second
equation by $\bar{\Psi },$ and integrating over $\mathbb{R},$ we obtain
\begin{equation}\label{ODE2}
\begin{aligned}
\int_{-\infty }^{\infty }\left( |\Phi ^{\prime
}|^{2}-\alpha|\Phi|^{p}-\tau|\Phi
|^{q}|\Psi|^{q}\right) \ dx=-\omega
_{1}\int_{-\infty }^{\infty }|\Phi|^{2}\
dx=-\omega _{1}s, \\
\int_{-\infty }^{\infty }\left( |\Psi ^{\prime
}|^{2}-\beta|\Psi|^{r}-\tau|\Phi
|^{q}|\Psi|^{q}\right) \ dx=-\omega
_{2}\int_{-\infty }^{\infty }|\Psi|^{2}\
dx=-\omega _{2}t.
\end{aligned}
\end{equation}
From Lemma~\ref{negative}, applied to $(f_{n},g_{n})=(\Phi,\Psi),$ we have that $\omega_{1}, \omega_{2}>0.$
This
proves assertion (b) of Theorem~\ref{existence}.

\smallskip

We now prove statement (c) of Theorem~\ref{existence}. We write
\begin{equation*}
\Phi (x)=e^{i\theta _{1}(x)}\left\vert \Phi (x)\right\vert \ \text{ and }\ \Psi
(x)=e^{i\theta _{2}(x)}\left\vert \Psi (x)\right\vert ,
\end{equation*}
where $\theta _{1},\theta _{2}:\mathbb{R}\to \mathbb{R}.$ Define $\tilde{\phi }(x)=\left\vert
\Phi (x)\right\vert $ and $\tilde{\psi }(x)=\left\vert \Psi
(x)\right\vert .$ Note that $(\tilde{\phi },\tilde{\psi })$ is also
in $\mathcal F_{s,t}$, as follows from Lemma~\ref{posmin}. Therefore, $(\tilde{\phi },\tilde{\psi })$
satisfies the Lagrange multiplier equations
\begin{equation}\label{ODE4}
\left\{
\begin{aligned}
-\tilde{\phi} ^{\prime \prime }+\omega_{1}\tilde{\phi} &=\alpha|\tilde{\phi}|^{p-2}\tilde{\phi}+\tau|\tilde{\psi}|^{q}|\tilde{\phi}|^{q-2}\tilde{\phi}, \\
-\tilde{\psi} ^{\prime \prime }+\omega_{2}\tilde{\psi} &=\beta|\tilde{\psi}|^{r-2}\tilde{\psi}+\tau|\tilde{\phi}|^{q}|\tilde{\psi}|^{q-2}\tilde{\psi},
\end{aligned}
\right.
\end{equation}
(The Lagrange multipliers are
determined by the equation \eqref{ODE2} and this equation stay same when $(\Phi ,\Psi
)$ is replaced by $(\tilde{\phi },\tilde{\psi }),$ and hence the Lagrange multipliers are unchanged.)
We compute
\begin{equation}\label{pos1}
\Phi ^{\prime \prime }=e^{i\theta _{1}}\left( \omega _{1}\tilde{\phi }%
-\alpha|\tilde{\phi }| ^{p-2}\tilde{\phi }-\tau|
\tilde{\psi}|^{q}|\tilde{\phi}|^{q-2}\tilde{\phi }-(\theta _{1}^{\prime })^{2}\tilde{%
\phi }+2i\theta _{1}^{\prime }\tilde{\phi }^{\prime }+i\theta
_{1}^{\prime \prime }\tilde{\phi }\right) .
\end{equation}
On the other hand, from the first equation of \eqref{ODE1}, we have that
\begin{equation}\label{pos2}
\Phi ^{\prime \prime }=e^{i\theta _{1}}\left( \omega _{1}\tilde{\phi }%
-\alpha|\tilde{\phi }|^{p-2}\tilde{\phi }-\tau|
\tilde{\psi}|^{q}|\tilde{\phi}|^{q-2}\tilde{\phi }\right) .
\end{equation}

From \eqref{pos1} and \eqref{pos2} , we obtain
\begin{equation*}
(\theta _{1}^{\prime }(x))^{2}\tilde{\phi }(x)-2i\theta _{1}^{\prime
}(x)\tilde{\phi }^{\prime }(x)-i\theta_{1}^{\prime \prime}(x)%
\tilde{\phi }(x)=0.
\end{equation*}
Equating the real part of the last equation, we conclude that $\theta
_{1}^{\prime }(x)=0,$ and hence $\theta _{1}(x)$ is constant. Similarly, $%
\theta _{1}(x)$ is constant.

\smallskip

Next, for any $\xi>0,$ define the function $K_\xi(x)$ by
\begin{equation*}
K_{\xi}(x)=\frac{1}{2\sqrt{\xi}}e^{-\sqrt{\xi} |x|}.
\end{equation*}
A calculation using Fourier transform shows that the operators $\omega_1-\partial_{xx}$ and $\omega_2-\partial_{xx}$ appearing
in \eqref{ODE4} are invertible on $H^1,$ with inverse given by convolution with
the functions $K_{\omega_1}$ and $K_{\omega_2}$ respectively.
Then, the Lagrange multiplier equations associated with $(\tilde{\phi}, \tilde{\psi})$ can be written as
\begin{equation*}
\label{convfortildephi} \tilde\phi=K_{\omega_1}
\star\left(\alpha|\tilde{\phi}|^{p-2}\tilde{\phi}+\tau|\tilde{\psi}|^{q}|\tilde{\phi}|^{q-2}\tilde{\phi}\right),\
\tilde\psi=K_{\omega_2}
\star\left(\beta|\tilde{\psi}|^{r-2}\tilde{\psi}+\tau|\tilde{\phi}|^{q}|\tilde{\psi}|^{q-2}\tilde{\psi}\right)
\end{equation*}
Since the convolutions of $K_{\omega_1}$ and $K_{\omega_2}$ with functions that are
 everywhere non-negative and not
identically zero must produce everywhere positive functions, it
follows that $\tilde \phi(x)>0$ and $\tilde \psi(x)>0$ for
all $x \in \mathbb R$. This completes proof of statement Theorem~\ref{existence}(c).

\smallskip

It remains to prove part (d) of Theorem~\ref{existence}.
Suppose to the contrary that $\mathcal{F}_{s,t}$ is unstable. Then there exist a
number $\epsilon >0,$ a sequence of times ${t_{n}},$ and a sequence
${(u_n(x,0),v_n(x,0))}$ in $Y$ such that for all $n,$
\begin{equation}
\label{idconvtoF} \inf\{\|(u_n(x,0),v_n(x,0))-(\Phi,\Psi)\|_Y : (\Phi,\Psi)\in
\mathcal{F}_{s,t}\}<\frac{1}{n};
\end{equation}
and
\begin{equation}\label{contra}
\inf\{\|(u_n(\cdot,t_n),v_n(\cdot,t_n)-(\Phi,\Psi)\|_Y : (\Phi,\Psi)\in
\mathcal{F}_{s,t}\} \geq \epsilon,
\end{equation}
where $(u_n(x,t),v_n(x,t))$ solves \eqref{CNLS} with initial data $(u_n(x,0),v_n(x,0)).$
 From \eqref{idconvtoF} and the continuity of the functionals $H$ and $Q,$
we have
\begin{equation} \label{initdata}
\begin{aligned}
\lim_{n\to \infty }H(u_n(x,0),v_n(x,0))&= \Theta(s,t),\\
\lim_{n\to \infty }Q(u_n(x,0)) &=s,\\
\lim_{n\to \infty }Q(v_n(x,0))&=t.
\end{aligned}
\end{equation}
Denote $R_n=u_n(\cdot,t_n)$ and $S_n=v_n(\cdot,t_n).$
Since $H(u,v)$ and $Q(u)$ are conserved quantities,
then \eqref{initdata} implies
\begin{equation*}
\begin{aligned}
\lim_{n \to \infty }H(R_n,S_n)&=\Theta(s,t),\\
\lim_{n \to \infty }Q(R_n)&=s,\\
\lim_{n\to \infty }Q(S_n)&=t.
\end{aligned}
\end{equation*}
Therefore $\{(R_{n},S_{n})\}$ is a minimizing sequence for $\Theta(s,t).$
 Now, by the first part of Theorem \ref{existence}, there
exists a subsequence $\{(R_{n_k},S_{n_k})\}$,
$\{y_k\}\subset \mathbb{R}$, and a pair $(\Phi,\Psi) \in
\mathcal{F}_{s,t}$ such that
\begin{equation}
\lim_{k \to \infty}
\|(R_{n_k}(\cdot+y_k),S_{n_k}(\cdot+y_k))-(\Phi,\Psi)\|_Y=0.
\end{equation}
Then, for some sufficiently large $k$,
\begin{equation*}
\|(R_{n_k}(\cdot+y_k),S_{n_k}(\cdot+y_k))-(\Phi,\Psi)\|_Y<\epsilon,
\end{equation*}
and hence
\begin{equation}
\label{penult}
\|(R_{n_k},S_{n_k})-(\Phi(\cdot-y_k),\Psi(\cdot-y_k))\|_Y<\epsilon.
\end{equation}
Since $\mathcal{F}_{s,t}$ is invariant under translations, $(\Phi(\cdot-y_k),\Psi(\cdot-y_k))$ belongs to
$\mathcal{F}_{s,t}$, contradicting
\eqref{contra}, and hence the minimizing set $\mathcal
F_{s,t}$ must be stable.
\end{proof}

\end{document}